\documentclass[oneside,10pt]{article}          
\usepackage[b5paper]{geometry}	    
\usepackage{amsfonts,amsmath,latexsym,amssymb} 
\usepackage{amsthm}                
\usepackage{mathrsfs,upref}         
\usepackage{mathptmx}		    
\usepackage{jmi}	            

\usepackage[T1]{fontenc}
\usepackage{hyperref}
\usepackage[nameinlink,noabbrev]{cleveref}

\usepackage{xcolor}

\usepackage{enumitem}
\newtheorem{thm}{Theorem}

\newtheorem{prop}{Proposition}
\newtheorem{cor}[thm]{Corollary}
\newtheorem{example}[thm]{Example}
\newtheorem{rem}{Remark}
\numberwithin{equation}{section}
\Crefname{prop}{Proposition}{Propositions}

\begin{document}

\title{Weighted Hardy Inequalities for Nested Averages}

\author{Ludovick Bouthat and Pierre-Olivier Parisé}

\address{Ludovick Bouthat, 2325 Rue de l'Université, Québec, QC G1V 0A6\\
\email{ludovick.bouthat.1@ulaval.ca}}

\address{Pierre-Olivier Parisé, 3351 Bd des Forges, Trois-Rivières, QC G8Z 4M3,\\
\email{pierre-olivier.parise@uqtr.ca}}

\CorrespondingAuthor{Ludovick Bouthat}

\thanks{POP was supported by an NSERC Discovery Grant No. RGPIN-2026-04740.}

\date{DD.MM.YYYY}                               



\begin{abstract}
We study a family of Hardy-type inequalities for weighted averages over nested subsets of a measure space. Given a partition of a measure space and a weight function $m$, we consider operators of the form
\[
f \mapsto \frac{1}{M_n}\int_{X^{(n)}} m(x)f(x)\,\mathrm{d}\mu(x),
\]
with additional weights on the resulting sequence of averages. In particular, we generalize an inequality obtained by Vincent and Sohani in \cite{VincentSohani2025} and characterize the boundedness in terms of the finiteness of a single testing quantity $\beta$. We also provide two-sided estimates for the best constant $C_{\mathrm{opt}}$, namely
\[
\beta \leq C_{\mathrm{opt}} \leq p^{1/q} (p')^{1/p'}\beta \leq 2\beta.
\]
Thus the characterization is never off by more than a factor of 2. We also develop a second approach, inspired by Broadbent's proof of Hardy's inequality, which gives a local sufficient condition that often provides sharper constants and recovers several important cases, including the classical weighted Hardy inequality.
\end{abstract}

\maketitle

\section{Introduction}

In his work on integral equations, Hilbert proved that if $(a_n)_{n \geq 1}$ is a sequence of non-negative terms in $\ell^2$, then the double series
\[
\sum_{m=1}^{\infty}\sum_{n=1}^{\infty}
    \frac{a_m a_n}{m+n}
\]
is convergent \cite{Hilbert1906}. More precisely, Hilbert showed that this series is bounded by a constant multiple of $\sum_{n=1}^{\infty} a_n^2$. The sharp form of this result, with best constant $\pi$, was later obtained by Schur \cite{Schur1911}.

It was in the search for a simpler proof of Hilbert's inequality that Hardy was led to what is now known as Hardy's inequality \cite{hardy1919notes}. Hardy observed that Hilbert's inequality could be derived from the convergence of the sequence of arithmetic means: if $(a_n)_{n \geq 1} \in \ell^2$, then
\[
\sum_{n=1}^{\infty} \left( \frac{a_1+\cdots+a_n}{n} \right)^{\!2} \leq 4\sum_{n=1}^{\infty} |a_n|^2.
\]
Although this result first appeared as an auxiliary step, the mathematical community quickly recognized that it was of independent interest. This insight marked the beginning of a long and influential line of research centered on inequalities controlling averages of sequences or functions by their original norms. See for example the papers of T. A. A. Broadbent \cite{MR1574000}, E. B. Elliott \cite{MR1574962}, K. Grandjot \cite{MR1574375}, G. H. Hardy \cite{hardy1925notes}, T. Kaluza and G. Szeg\"{o} \cite{MR1574830}, and K. Knopp \cite{MR1574165}.

After Hardy communicated the above result to Marcel Riesz, the latter gave an alternative proof, published in a paper of the former \cite{MR1544414}, which yielded the following $p$-version: for $p>1$, $a_n \geq 0$, and $(a_n)_{n \geq 1} \in \ell^p$, one has
\[
\sum_{n=1}^{\infty} \left( \frac{a_1+\cdots+a_n}{n} \right)^p \leq C_p \sum_{n=1}^{\infty} a_n^p .
\]
The optimal value of the constant was later shown to be $C_p = \big( \frac{p}{p-1} \big)^p$, a result attributed to Landau, who supplied Hardy with a proof of the sharpness of this constant \cite{MR1575105}. In \cite{MR1544414}, Hardy also recorded the corresponding integral formulation, without proof, which bounds the $L^p$-norm of the averaging operator
\[
f \mapsto \frac{1}{x}\int_a^x f(t)\,\mathrm{d}t
\]
by the $L^p$-norm of $f$.

Recently, Bouthat and Mashreghi studied the boundedness as an operator on $\ell^2$ of the \emph{$L$-matrices} \cite{MRtmp}, which are defined according to a sequence $(a_n)_{n\geq 1}$ by setting
\[
L =[a_n]=
\begin{pmatrix}
a_0 & a_1 & a_2 & a_3 & \cdots \\
a_1 & a_1 & a_2 & a_3 & \cdots \\
a_2 & a_2 & a_2 & a_3 & \cdots \\
a_3 & a_3 & a_3 & a_3 & \cdots \\[-2pt]
\vdots & \vdots & \vdots & \vdots & \ddots \\
\end{pmatrix}.
\]
Since then, several authors studied these matrices (see \cite{MR4364583,MR4335803,MR4509958,MR4375658}). In this process, the inequality 
\begin{equation}\label{eq - initial}
    \sum_{n=1}^{\infty} \frac{1}{4^n} \left| \sum_{j=1}^{4^n} a_j \right|^2
\leq
C \sum_{n=1}^{\infty} |a_n|^2
\end{equation}
naturally appeared \cite{MR4364583}. This motivated Bouthat, Mashreghi and Morneau-Guérin to ask if this inequality was a particular instance of a more general family of inequalities. 

More specifically, the authors observed that the form of \eqref{eq - initial} is very reminiscent of Hardy's inequality in the case $p=2$, where the ordinary arithmetic means are replaced by more general weighted means taken over nested subsets of a sequence. In the above examples, the subsets of sequences are $N_n = \{1,2,\dots,4^n\}$. More generally, one may consider a partition of $\mathbb{N}$
\[
\mathbb{N} = N_1 \cup N_2 \cup \cdots
\]
and define
\[
\mathbf{N}_n := N_1\cup\cdots\cup N_n, \qquad n \geq 1.
\]
In that case, it was shown in \cite{Bouthat2023} that if $(a_n)_{n \geq 1}$ is a sequence of complex numbers, $(m_n)_{n \geq 1}$ and $(M_n)_{n \geq 1}$ are two sequences of weights, $1<p<\infty$, and\vspace{-2pt}
\begin{equation}\label{E:3}
\rho := \sup_{n \geq 1} \, \left(w_n \sum_{k=n}^{\infty}\frac{(w_1+\cdots+w_k)^{p-1}}{M_k^p}\right)^{\!\!1/p} \!< \infty,\vspace{-2pt}
\end{equation}
where
\begin{equation*}
w_n := \left(\sum_{j \in N_n} m_j^{p'} \right)^{\!\!1/p'}\!, \qquad n \geq 1\vspace{-2pt}
\end{equation*}
and $p'=\frac{p}{p-1}$, then\vspace{-4pt}
\[
\left(\sum_{n=1}^{\infty} \bigg|\frac{1}{M_n}\sum_{j \in \mathbf{N}_n} m_ja_j \bigg|^p\right)^{\!\!1/p}
\!\!\leq \rho  \left(\sum_{n=1}^{\infty} |a_n|^p\right)^{\!\!1/p}.
\]

Following this work, in 2025, Vincent and Sohani generalized this inequality even more to integral operators. In particular they proved the following result.

\begin{thm}[\cite{VincentSohani2025}, Theorem 2]\label{thm - VincentSohani}
Let $A_1,A_2,\ldots$ be a partition of $\mathbb{R}^m$ such that each $A_i$ is measurable. Let $B_n = \cup_{i=1}^n A_i$ and let $p,q \in (1,\infty)$ be such that $\frac{1}{p}+\frac{1}{p'}=1$. For $g>0$ defined on $\mathbb{R}^m$, define \vspace{-2pt}
\[ 
w_i = \left( \int_{A_i} g(x)^{p'} \,\mathrm{d}x \right)^{\!\smash{1/p'}}. 
\] 
If $M_n$ is a sequence of positive numbers such that 
\[ 
\rho := \sup_{n \geq 1} \, \left(w_n \sum_{k=n}^{\infty}\frac{(w_1+\cdots+w_k)^{p-1}}{M_k^p}\right)^{\!\!1/p} < \infty, 
\] 
then 
\[ 
\left( \sum_{n=1}^{\infty} \left| \frac{1}{M_n} \int_{B_n} g(x) f(x) \,\mathrm{d}x \right|^p \right)^{\!\!1/p} \!\leq \rho \left( \int_{\mathbb{R}^m} |f(x)|^p \,\mathrm{d}x \right)^{\!\!1/p}. 
\] 
\end{thm}

While Vincent and Sohani considerably generalized the inequality of Bouthat, Mashreghi and Morneau-Guérin, there are still some limitations to both of them. For instance, one can show that, while both of these inequalities look similar to the classical Hardy inequality, neither of them allows to recover this simple case. Moreover, while they are enough to obtain \eqref{eq - initial} with a certain constant $C>0$, they are not sufficiently precise in this regime to identify the sharp multiplicative constant.

More generally, neither of the sufficient conditions for these inequalities is also necessary. Lastly, one may notice that there is still room for apparent generalization. In particular, one may consider a general measure space $(X,\mu)$, consider mixed norms, or add another sequence of weights in the inequality, which would allow to recover more past generalizations of Hardy's inequality in the literature (see \Cref{sec - local}). Note that here and throughout, a sequence of weights is meant to be any sequence of positive numbers. It is the objective of this paper to simultaneously address all of these concerns.

Hence, the main question considered in this paper is the following: \emph{Let $1 \leq p\leq q\leq \infty$. Let $(X,\mu)$ be a measure space. Let $X = X_1 \cup X_2 \cup \cdots$ be a partition of $X$ and set $X^{(n)} := X_1 \cup \cdots \cup X_n$. Let $(b_n)_{n\geq 1}$ and $(M_n)_{n\geq 1}$ be sequences of weights. Let $m$ be a positive measurable function on $X$ and $f$ a complex measurable function on $X$. Under which conditions is there a number $\rho>0$ such that
\begin{equation}\label{eq - goal}
    \left(\sum_{n = 1}^\infty b_n \left| \frac{1}{M_n} \int_{X^{(n)}} m (x) f (x)  \mathrm{d}\mu (x) \right|^q\right)^{\!\!1/q} \!\!\leq \rho \left(\sum_{n = 1}^\infty b_n \int_{X_n} |f (x)|^p  \mathrm{d}\mu (x) \right)^{\!\!1/p}.
\end{equation}
}

The paper is outlined as follows. In \Cref{sec - charac}, we begin by showing that, while the generalizations proposed above are convenient and interesting, they can be equivalently reduced to the discrete case. As a corollary, we are able to completely characterize the condition under which a constant $\rho>0$ exists such that \eqref{eq - goal} holds true by using a 1978 result of Bradley \cite{Bradley}. Some examples are then studied to compare the sharpness of the constant $\rho$ obtained from this approach to the constants obtained in \cite{Bouthat2023} and \cite{VincentSohani2025}.

While this proposed approach provides a complete characterization, we shall see that the constants obtained are often not sharp. Consequently, in \Cref{sec - local}, we will explore another approach which gives a local sufficient condition for the validity of \eqref{eq - goal}. While the locality of the condition is a strong hypothesis to satisfy, we shall see that it often comes with the benefit of giving some sharper constants.

Lastly, we conclude in \Cref{sec - conclusion} by raising some open questions and research avenues for these kinds of inequalities.

\section{A complete characterization}\label{sec - charac}

As outlined above, we first prove an equivalence between the measure-theoretic case and the discrete case. Note that here and in the rest of the paper, we write $p'$ to denote the Hölder conjugate of $p$, i.e., the number satisfying $\frac{1}{p}+\frac{1}{p'}=1$, and we define
\begin{equation}\label{def - wn}
    w_n := \begin{cases}
        \left( \int_{X_n} |m(x)|^{p'}\,\mathrm{d}\mu(x) \right)^{\!\smash{1/p'}} \quad&\text{if }p>1,\\
        \operatorname*{ess\,sup}_{x \in X_n} |m(x)|\quad&\text{if }p=1.
    \end{cases}
\end{equation}



\begin{prop}\label{prop - equiv}
Let $1\leq p\leq q\leq\infty$. Let $(X,\mu)$ be a measure space, with $\mu$ semi-finite if $p=1$. Let $X = X_1 \cup X_2 \cup \cdots$ be a partition of $X$ and set $X^{(n)} := X_1 \cup \cdots \cup X_n$. Let $(b_n)_{n\geq 1}$ and $(M_n)_{n\geq 1}$ be sequences of weights. Let $m$ be an $L^{p'} (X)$ function and define $w_n$ as in \eqref{def - wn}. Then there exists $\rho>0$ such that the inequality
\begin{equation}\label{eq - measure-hardy}
    \left(\sum_{n=1}^{\infty} b_n \left| \frac{1}{M_n} \int_{X^{(n)}}m(x)f(x)\,\mathrm{d}\mu(x) \right|^q\right)^{\!\!1/q}\! \leq \rho \left(\sum_{n=1}^{\infty} b_n \int_{X_n}|f(x)|^p\,\mathrm{d}\mu(x) \right)^{\!\!1/p}
\end{equation}
holds for every measurable function $f$ if and only if the discrete inequality
\begin{equation}\label{eq - discrete-hardy}
    \left(\sum_{n=1}^{\infty} b_n \bigg|\frac{1}{M_n} \sum_{k=1}^{n}w_k a_k \bigg|^q \right)^{\!\!1/q}\! \leq \rho \left(\sum_{k=1}^{\infty} b_k |a_k|^p\right)^{\!\!1/p}
\end{equation}
holds for every sequence $a=(a_k)_{k\geq 1}$.
\end{prop}

\begin{rem}\label{rem - explain}
    Here, the case $p=\infty$ and $q=\infty$ should be interpreted in the natural way. Namely, the left and right-hand sides of \eqref{eq - measure-hardy} should be understood respectively as 
    \[
    \sup_{b_n>0} \left| \frac{1}{M_n} \int_{X^{(n)}}m(x)f(x)\,\mathrm{d}\mu(x) \right| \qquad\text{and}\qquad \operatorname*{ess\,sup}_{\cup_{b_n> 0} X_n}|f(x)|,
    \]
    while the left and right-hand sides of \eqref{eq - discrete-hardy} should be interpreted respectively as
    \[
    \sup_{b_n>0} \bigg|\frac{1}{M_n} \sum_{k=1}^{n}w_k a_k \bigg| \qquad\text{and}\qquad \sup_{b_n>0}|a_n|.
    \]
\end{rem}

\begin{proof}
We first prove that \eqref{eq - discrete-hardy} implies \eqref{eq - measure-hardy}. For each $k\geq 1$, define
\[
F_k:= 
\begin{cases}
    \left( \int_{X_k} |f(x)|^{p}\,\mathrm{d}\mu(x) \right)^{\!\smash{1/p}} \quad&\text{if }p<\infty,\\
    \operatorname*{ess\,sup}_{x \in X_k} |f(x)|\quad&\text{if }p=\infty.
\end{cases}
\]
By Hölder's inequality,
\[
\left| \int_{X_k}m(x)f(x)\,\mathrm{d}\mu(x) \right| \leq \left( \int_{X_k}|m(x)|^{p'}\,\mathrm{d}\mu(x) \right)^{\!\!1/p'} \!\left( \int_{X_k}|f(x)|^p\,\mathrm{d}\mu(x) \right)^{\!\!1/p} \!= w_k F_k,
\]
with the obvious modifications in the inequality if $p=1$ or $\infty$. Therefore,
\[
\left| \int_{X^{(n)}}m(x)f(x)\,\mathrm{d}\mu(x) \right| = \left| \sum_{k=1}^{n} \int_{X_k}m(x)f(x)\,\mathrm{d}\mu(x) \right| \leq \sum_{k=1}^{n}w_k F_k .
\]
Applying \eqref{eq - discrete-hardy} to the sequence $(F_k)_{k\geq 1}$ gives
\begin{align*}
    \left(\sum_{n=1}^{\infty} b_n \bigg| \frac{1}{M_n} \int_{X^{(n)}}m(x)f(x)\,\mathrm{d}\mu(x) \bigg|^q\right)^{\!\!1/q}\!\!
    &\leq \left(\sum_{n=1}^{\infty} b_n \bigg| \frac{1}{M_n} \sum_{k=1}^{n}w_k F_k \bigg|^q\right)^{\!\!1/q}\\
    &\leq \rho \!\left( \sum_{k=1}^{\infty}b_k F_k^p \right)^{\!\!1/p}\!, 
    %
\end{align*}
which is \eqref{eq - measure-hardy}.

Conversely, suppose that \eqref{eq - measure-hardy} holds. We prove \eqref{eq - discrete-hardy}. If $w_k=0$, then the $k$th coordinate does not contribute to the left-hand side of \eqref{eq - discrete-hardy}; hence we may assume $a_k=0$ whenever $w_k=0$. We write $W := \{ n \in \mathbb{N} \, : \, w_n > 0 \}$.

Let $k \in W$. If $p=1$, then $w_k=\operatorname*{ess\,sup}_{X_k}|m|$. Hence, since $\mu$ is semi-finite, choose a set $E_{k,\varepsilon}\subset X_k$ of finite positive measure on which $|m|>w_k-\varepsilon$. Now define $f_{k}$ on $X_k$ by
\[
f_{k} := \begin{cases} 
    a_k e^{-i \arg(m)} \frac{\chi_{E_{k,\varepsilon}} }{\mu(E_{k,\varepsilon})} & \text{ if } p = 1, \\
    a_k w_k^{1-p'}e^{-i \arg(m)}|m|^{p'-1} \chi_{X_k} & \text{ if } p > 1,    
    \end{cases}
\qquad 
\]
where $\arg(m)$ is arbitrary when $m=0$, and set $f := \sum_{k \in W} f_{k}$. If $p=1$, then
\[
\int_{X_k}|f(x)|^p\,\mathrm{d}\mu(x) = \int_{X_k}|a_k|\frac{\chi_{E_{k,\varepsilon}}}{\mu(E_{k,\varepsilon})}\,\mathrm{d}\mu(x) = |a_k|.
\]
If $1<p<\infty$, then
\[
\begin{aligned}
    \int_{X_k}|f(x)|^p\,\mathrm{d}\mu(x)
    %
    &= |a_k|^p w_k^{p(1-p')} \int_{X_k}|m(x)|^{p'}\,\mathrm{d}\mu(x) 
    %
    = |a_k|^p .
\end{aligned}
\]
Hence, in both cases, we have
\[
\left(\sum_{n=1}^{\infty} b_n \int_{X_n}|f(x)|^p\,\mathrm{d}\mu(x) \right)^{\!\!1/p} = \left(
\sum_{k=1}^{\infty} b_k |a_k|^p \right)^{\!\!1/p}.
\]
Lastly, if $p=\infty$, then we verify directly that
\[
\left(\sum_{n=1}^{\infty} b_n \int_{X_n}|f(x)|^p\,\mathrm{d}\mu(x) \right)^{\!\!1/p} =\operatorname*{ess\,sup}_{\cup_{b_n> 0} X_n}|f(x)| = \sup_{b_n> 0}|a_n|
= \left(\sum_{k=1}^{\infty} b_k |a_k|^p \right)^{\!\!1/p}.
\]
Moreover, we always have
\[
\int_{X_k}m(x)f(x)\,\mathrm{d}\mu(x) = \lambda_{k,\varepsilon}a_k,
\]
where $w_k-\varepsilon \leq \lambda_{k,\varepsilon} \leq w_k$ if $p=1$ and $\lambda_{k,\varepsilon}=w_k$ if $p>1$. In the case $p = 1$, applying \eqref{eq - measure-hardy}, we find 
\begin{align*}
    \left( \sum_{n=1}^{\infty} b_n \bigg| \frac1{M_n} \sum_{k=1}^{n} \lambda_{k,\varepsilon}a_k \bigg|^q \right)^{\!\!1/q}\!\! &= \left(\sum_{n=1}^{\infty} b_n \left| \frac{1}{M_n} \int_{X^{(n)}}m(x)f(x)\,\mathrm{d}\mu(x) \right|^q\right)^{\!\!1/q} \\
    &\leq \rho \!\left(\sum_{n=1}^{\infty} b_n \int_{X_n}|f(x)|^p\,\mathrm{d}\mu(x) \right)^{\!\!1/p} \!\!= \rho\!\left(\sum_{k=1}^{\infty} b_k |a_k|^p \right)^{\!\!1/p}\!.
\end{align*}
The result then follows directly by letting $\varepsilon\to0$ and using Fatou's lemma. Note that the case $q=\infty$ is handled in the same way, with the weighted $\ell^q$-norm replaced by the supremum norm on the indices satisfying $b_n>0$.
\end{proof}

The preceding proposition shows that the measure-theoretic problem considered above is equivalent to a discrete  inequality. The corresponding continuous problem has a classical answer, due to Muckenhoupt \cite{Muckenhoupt}. More precisely, if $1<p<\infty$, then
\[
\left(\int_0^\infty \left|u(x)\int_0^x f(t)\,\mathrm{d}t \right|^p \,\mathrm{d}x \right)^{\!\!1/p}\!
\le C \left( \int_0^\infty |v(x)f(x)|^p\,\mathrm{d}x \right)^{\!\!1/p}
\]
holds for all measurable functions $f$ if and only if
\[
A_p:= \sup_{s>0} \left(\int_s^\infty |u(x)|^p\,\mathrm{d}x\right)^{\!\!1/p}\! \left(\int_0^s |v(x)|^{-p'}\,\mathrm{d}x\right)^{\!\!1/p'} <\infty .
\]
Moreover, if $C_{\mathrm{opt}}$ is the best possible constant, then $A_p\le C_{\mathrm{opt}} \le p^{1/p}(p')^{1/p'}A_p .$ 
Bradley \cite{Bradley} later extended this result to mixed norms. Namely, for $1\le p\le q\le\infty$, the inequality\vspace{-1pt}
\[
\left(\int_0^\infty\left|u(x)\int_0^x f(t)\,\mathrm{d}t\right|^q \mathrm{d}x\right)^{\!\!1/q} \!\le C\left(\int_0^\infty |v(x)f(x)|^p\,\mathrm{d}x\right)^{\!\!1/p}
\]
is characterized by the same type of testing condition. That is,\vspace{-1pt}
\[
A_{p,q}:= \sup_{s>0} \left(\int_s^\infty |u(x)|^q\,\mathrm{d}x\right)^{\!\!1/q} \!\left(\int_0^s |v(x)|^{-p'}\,\mathrm{d}x\right)^{\!\!1/p'}\! <\infty ,
\]
with the usual modifications when $p=1$ or $q=\infty$. In particular, the case $q=p$ recovers Muckenhoupt's theorem.

The results of Muckenhoupt and Bradley are formulated for integral operators on $(0,\infty)$. In the present paper, however, \Cref{prop - equiv} reduces the problem to a discrete operator. The corresponding discrete criterion may be obtained from the continuous one by a standard step-function discretization, or by repeating the proof with sums in place of integrals. We state this discrete form explicitly in the notation needed below.

\begin{prop}[Discrete Muckenhoupt--Bradley criterion]\label{prop - old}
Let $1\leq p\leq q\leq \infty$, and let
$(u_n)_{n\ge1}$ and $(v_n)_{n\ge1}$ be sequences of weights. Then there exists some $C>0$ such that\vspace{-1pt}
\[
\left(\sum_{n=1}^{\infty} \bigg| u_n\sum_{k=1}^n a_k \bigg|^q \right)^{\!\!1/q}\!\! \le C \left(\sum_{k=1}^{\infty}|v_ka_k|^p \right)^{\!\!1/p}
\]
holds for every sequence $a=(a_k)_{k\ge1}$ if and only if\vspace{-1.5pt}
\[
\beta := \sup_{N\ge1} \left(\sum_{n=N}^{\infty} |u_n|^q\right)^{\!\!1/q}\! \left(\sum_{k=1}^N v_k^{-p'}\right)^{\!\!1/p'}\!
<\infty
\]
\!$($\!with the obvious modifications if $p,q\in\{1,\infty\}$, as in \Cref{rem - explain}\!$)$. Moreover, if $C_{\mathrm{opt}}$ denotes the best possible constant, then $\beta \le C_{\mathrm{opt}} \le p^{1/q}(p')^{1/p'}\beta .$
\end{prop}

\begin{rem}
    If $p=1$, then $p^{1/q}(p')^{1/p'}=1$. Hence, the optimal constant is \vspace{-1.5pt}
    \begin{equation}\label{eq - beta}
        \beta = \sup_{N\ge1} \max_{1\leq k \leq N} \frac{1}{v_k}\left(\sum_{n=N}^{\infty} |u_n|^q\right)^{\!\!1/q}.
    \end{equation}
    However, this constant may be replaced with an equivalent quantity, namely\vspace{-1.5pt}
    $$
        \beta' := \sup_{N \geq 1} \frac{1}{v_N} \left( \sum_{n = N}^\infty |u_n|^q \right)^{1/q}, 
    $$
    which is easier to handle in applications than \eqref{eq - beta}. Indeed, the inequality $\beta'\le\beta$ is immediate. Conversely, if $k\le N$, then $\left(\sum_{n=N}^{\infty} |u_n|^q\right)^{1/q}\le \left(\sum_{n=k}^{\infty} |u_n|^q\right)^{1/q}$, and therefore
    \[
    \frac{1}{v_k}\left(\sum_{n=N}^{\infty} |u_n|^q\right)^{\!\!1/q}
    \le
    \frac{1}{v_k}\left(\sum_{n=k}^{\infty} |u_n|^q\right)^{\!\!1/q}
    \le
    \beta' .
    \]
    Maximizing over $1\le k\le N$ and then over $N$ yields $\beta\le\beta'$. Hence $\beta=\beta'$.
\end{rem}

Combining \Cref{prop - equiv,prop - old} immediately gives the following corollary.

\begin{cor}\label{cor - charac}
Let $1\leq p\leq q\leq\infty$. Let $(X,\mu)$ be a measure space, with $\mu$ semi-finite if $p=1$. Let $X = X_1 \cup X_2 \cup \cdots$ be a partition of $X$ and set $X^{(n)} := X_1 \cup \cdots \cup X_n$. Let $(b_n)_{n\geq 1}$ and $(M_n)_{n\geq 1}$ be sequences of weights. Let $m$ be a measurable function on $X$ and define $w_n$ as in \eqref{def - wn}. Then
\begin{equation*}\label{eq - charac-bound}
    \left(\sum_{n=1}^{\infty} b_n \left| \frac{1}{M_n} \int_{X^{(n)}}m(x)f(x)\,\mathrm{d}\mu(x) \right|^q \right)^{\!\!1/q}\!\! \leq p^{1/q}(p')^{1/p'} \beta \left( \sum_{n=1}^{\infty} b_n \int_{X_n}|f(x)|^p\,\mathrm{d}\mu(x) \right)^{\!\!1/p}
\end{equation*}
if and only if
\[
\beta = \sup_{N\ge1} \left(\sum_{n=N}^{\infty} \frac{b_n}{M_n^q}\right)^{\!\!1/q}\! \left(\sum_{k=1}^N \frac{w_k^{p'}}{b_k^{p'/p}}\right)^{\!\!1/p'}\!
<\infty
\]
\!$($\!with the obvious modifications if $p,q\in\{1,\infty\}$, as in \Cref{rem - explain}\!$)$.
\end{cor}

\begin{proof}
    By \Cref{prop - equiv}, the inequality is equivalent to the discrete inequality 
    \begin{equation}\label{eq - simple}
         \left(\sum_{n=1}^{\infty} \left|\frac{b_n^{1/q}}{M_n} \sum_{k=1}^{n}w_k a_k \right|^q \right)^{\!\!1/q}\! \leq \rho \left(\sum_{k=1}^{\infty} \big|b_k^{1/p}a_k\big|^p\right)^{\!\!1/p}.
    \end{equation}
    As we did in the proof of \Cref{prop - equiv}, observe that if $w_k= 0$, then the $k$th coordinate does not contribute to the left-hand side, and we may assume without any loss of generality that $a_k=0$. This is identical to just ignoring the $k$th summand in both the left and the right-hand sides, which is equivalent to simply assuming that $w_k\neq 0$. Hence, if we make this hypothesis, we can replace $a_k$ by $a_k/w_k$ and define 
    \[
    u_n=\frac{b_n^{1/q}}{M_n}
    \qquad\text{and}\qquad
    v_n=\frac{b_n^{1/p}}{w_n},
    \]
    to find that \eqref{eq - simple} is in fact equivalent to 
    \[
    \left(\sum_{n=1}^{\infty} \bigg| u_n\sum_{k=1}^n a_k \bigg|^q \right)^{\!\!1/q}\!\! \le \rho \left(\sum_{k=1}^{\infty}|v_ka_k|^p \right)^{\!\!1/p}.
    \]
    Hence, by \Cref{prop - old}, the desired inequality is equivalent to having
    \[
    \beta := \sup_{N\ge1} \left(\sum_{n=N}^{\infty} \frac{b_n}{M_n^q}\right)^{\!\!1/q}\! \left(\sum_{k=1}^N \frac{w_k^{p'}}{b_k^{p'/p}}\right)^{\!\!1/p'}\!
    <\infty .
    \]
    Moreover, we may choose $\rho = p^{1/q}(p')^{1/p'}\beta .$
\end{proof}

\begin{rem}
    In the important case where $b_n\equiv1$, for which the right-hand side of the inequality reduces to the $L^p(\mu)$-norm of $f$, the condition becomes simply
    \[
    \sup_{N\geq 1} \left( \sum_{n=N}^{\infty}\frac{1}{M_n^q} \right)^{\!\!1/q}\! \left(\sum_{k=1}^{N}w_k^{p'} \right)^{\!\!1/p'} \!< \infty.
    \]
\end{rem}

\Cref{cor - charac} characterizes the condition for which there is a $C>0$ such that
\[
\left(\sum_{n=1}^{\infty} b_n \left| \frac{1}{M_n} \int_{X^{(n)}}m(x)f(x)\,\mathrm{d}\mu(x) \right|^q \right)^{\!\!1/q}\!\! \leq C \left( \sum_{n=1}^{\infty} b_n \int_{X_n}|f(x)|^p\,\mathrm{d}\mu(x) \right)^{\!\!1/p}.
\]
However, this constant is not always optimal. Consider for instance the following cases.

\begin{enumerate}
    \item \emph{The classical discrete Hardy inequality}. Suppose that $X=\mathbb{N}$, $X_n=\{n\}$, $\mu$ is the counting measure, $p=q$, $m\equiv 1$, $b_n=1$ for all $n\geq 1$ and $M_n=n$. Then
    \[
    \beta = \sup_{N\geq 1} \bigg( \sum_{n=N}^{\infty}\frac{1}{n^p} \bigg)^{\!\!1/p}\! N^{1/q} = \sup_{N\geq 1} \bigg( N^{p-1}\sum_{n=N}^{\infty}\frac{1}{n^p} \bigg)^{\!\!1/p}\!.
    \]
    Now, define
    \[
    \varphi(x)=x\left(1-\Big(\frac{x}{x+1}\Big)^{p-1}\right).
    \]
    Then, with $y=(x+1)^{-1}$,
    \[
        \varphi'(x)=1-(1-y)^{p-1}(1+(p-1)y)\geq 0,
    \]
    since $(1-y)^{-a}\geq 1+ay$. Thus $\varphi$ is increasing, and for
    $n\geq N+1$,
    \[
    \frac{N((N+1)^{p-1}-N^{p-1})}{(N+1)^{p-1}} =\varphi(N)\leq \varphi(n) =n\!\left(1-\Big(\frac{n}{n+1}\Big)^{p-1}\right).
    \]
    Hence,
    \[
    \frac{(N+1)^{p-1}-N^{p-1}}{n^{p}} \leq \frac{(N+1)^{p-1}}{N} \left(\frac{1}{n^{p-1}}-\frac{1}{(n+1)^{p-1}}\right).
    \]
    Summing over $n\geq N+1$ gives
    \[
    ((N+1)^{p-1}-N^{p-1})\sum_{n=N+1}^{\infty}\frac{1}{n^{p}} \leq \frac{(N+1)^{p-1}}{N}\cdot\frac{1}{(N+1)^{p-1}} = \frac{1}{N}.
    \]
    Therefore, if $s_N:=N^{p-1}\sum_{n=N}^{\infty}\frac{1}{n^p}$, we find
    \[
    s_N-s_{N+1} = \frac{1}{N} - \bigl((N+1)^{p-1}-N^{p-1}\bigr) \sum_{n=N+1}^{\infty}\frac{1}{n^p} \geq 0,
    \]
    where $\zeta(\cdot)$ is the Riemann zeta function. Therefore, $s_N$ is decreasing and
    \[
    \beta =  \bigg(\sum_{n=1}^{\infty}\frac{1}{n^p} \bigg)^{\!\!1/p} \!= \zeta^{1/p}(p) < \infty.
    \]
    In particular, \Cref{cor - charac} yields 
    \[
    \sum_{n=1}^{\infty} \left( \frac{a_1+\cdots+a_n}{n} \right)^p \leq \left(\frac{p}{p-1}\right)^{\!p} (p-1)\zeta(p) \sum_{n=1}^{\infty} a_n^p,
    \]
    which is Hardy's inequality with the optimal constant multiplied by a factor of $(p-1)\zeta(p)$.

    \item \emph{The geometric Hardy inequality.} For every integer $b\geq 2$, \cite{Bouthat2023} shows that
    \[
    \sum_{k=1}^\infty \frac{1}{b^k} \bigg| \sum_{n=1}^{b^k} a_n \bigg|^2  \!\leq \frac{\sqrt{b}+1}{\sqrt{b}-1} \sum_{n=1}^\infty |a_n|^2 ,
    \]
    with the constant $\frac{\sqrt{b}+1}{\sqrt{b-1}} $ being optimal. With the corresponding choice of parameters, \Cref{cor - charac} gives the analogous inequality, but now with the constant 
    \[
    C = \frac{4b}{b-1}.
    \]
    Once again, the provided constant is not optimal.
\end{enumerate}

While \Cref{eq - charac-bound} is often not optimal, it is worth noting that it will never be far from optimal. Indeed, \Cref{prop - old} ensures that if $C_{\mathrm{opt}}$ denotes the optimal multiplicative constant for the inequality, then
\[
\beta \leq C_{\mathrm{opt}} \leq p^{1/q}(p')^{1/p'}\beta .
\]
Since for $1\leq p\leq q \leq \infty$, we have
\[
1 \leq p^{1/q}(p')^{1/p'} \leq p^{1/p}(p')^{1/p'} \leq 2,
\]
with the upper bound being attained only when $p=q=2$. Therefore, the bound of \Cref{cor - charac} can be off by at most a factor of 2.

\section{Local sufficient conditions}\label{sec - local}

As described above, \Cref{cor - charac} provides a complete characterization of the cases for which the inequality in \Cref{eq - charac-bound} holds. While we showed that the provided bound has a multiplicative constant that cannot be off by a factor of more than $2$ (and more precisely $p^{1/q}(p')^{1/p'}$), one may seek to find another approach for which the bounds obtained are often sharper, at least in the more important cases. 

Moreover, the condition in \Cref{cor - charac} is not always easy to verify. In this section, we provide another approach to treat the main inequality of this paper. This approach, based on Broadbent’s proof of Hardy's inequality \cite{MR1574000}, gives a simple sufficient condition to verify while also having the benefit of being sharp in several important cases.

\begin{thm}\label{thm - main}
    Let $1< p\leq q<\infty$. Let $(X,\mu)$ be a measure space. Let $X = X_1 \cup X_2 \cup \cdots$ be a partition of $X$ and set $X^{(n)} := X_1 \cup \cdots \cup X_n$. Let $(b_n)_{n\geq 1}$ and $(M_n)_{n\geq 1}$ be sequences of weights. Let $m$ be a measurable function on $X$ and define $w_n$ as in \eqref{def - wn}. Suppose that $w_n>0$ and set $\gamma_n:=b_n^{1/p+1/q'}\!\!/w_n$. If $0<\rho<\infty$ is such that
    \[
    \frac{1}{q}\frac{M_n(\gamma_n-\gamma_{n+1})}{b_n}
    +\frac{1}{q'}\frac{\gamma_n(M_n-M_{n-1})}{b_n}
    \geq
    \frac{1}{\rho}
    \qquad \text{for }~ n=1,2,3,\dots,
    \]
    then
    \begin{align*}
        \left(
        \sum_{n = 1}^\infty b_n
        \left| \frac{1}{M_n} \int_{X^{(n)}} m (x) f (x)
        \,\mathrm{d}\mu (x) \right|^q
        \right)^{\!\!1/q}
        \leq
        \rho
        \left(
        \sum_{n = 1}^\infty b_n \int_{X_n} |f (x)|^p
        \,\mathrm{d}\mu (x)
        \right)^{\!\!1/p}.
    \end{align*}
\end{thm}

\begin{proof}
    First observe that we always have
    \[
    \sum_{n = 1} b_n \left| \frac{1}{M_n} \int_{X^{(n)}} m (x) f (x) \mathrm{d}\mu (x) \right|^q \leq \sum_{n = 1}^\infty b_n \!\left( \frac{1}{M_n} \int_{X^{(n)}} |m (x) f(x)|  \mathrm{d}\mu (x) \right)^{\!q}.
    \]
    Hence, it is enough to prove the estimate with $m$ and $m$ replaced by $|m|$ and $|f|$, since $w_n$ is defined in terms of $|m|$ and the right-hand side depends only on $|f|$.
    Define $M_0=T_0:=0$, let $T_n := \frac{1}{M_n} \int_{X^{(n)}} |m(x) f (x)| \,\mathrm{d}\mu (x)$ for $n\geq 1$, and let
    \[
    \Delta_n := b_n T_n^q - \rho\gamma_n w_n\left(\int_{X_n} |f(x)|^p \,\mathrm{d}\mu(x)\right)^{\!\!1/p} T_n^{q-1}.
    \]
    Then Hölder's inequality implies that
    \begin{align*}
        \int_{X_n} |m(x) f(x)| \,\mathrm{d}\mu(x)
        &\leq \left(\int_{X_n} |m(x)|^{p'} \,\mathrm{d}\mu(x) \right)^{\!\!1/p'}\!
        \left(\int_{X_n} |f(x)|^p \,\mathrm{d}\mu(x) \right)^{\!\!1/p} \\
        &= w_n \left(\int_{X_n} |f(x)|^p \,\mathrm{d}\mu(x) \right)^{\!\!1/p}.
    \end{align*}
    It thus follows that
    \begin{align*}
        \Delta_n &= b_n T_n^q - \rho\gamma_nw_n\left(\int_{X_n} |f(x)|^p \,\mathrm{d}\mu(x)\right)^{\!\!1/p} T_n^{q-1} \\
        &\leq  b_n T_n^q - \rho\gamma_n \left( \int_{X_n} |m(x) f(x)| \,\mathrm{d}\mu(x) \right) T_n^{q-1} \\
        &= b_n T_n^q - \rho\gamma_n \big(M_nT_n-M_{n-1} T_{n-1}\big) T_n^{q-1} \\
        &= T_n^q \left( b_n - \rho\gamma_n M_n \right)
        + \rho\gamma_n M_{n-1} T_{n-1} T_{n}^{q-1}.
    \end{align*}
    But now recall Young's inequality which implies that
    \[
    T_{n-1} T_{n}^{q-1} \leq \frac{1}{q}\, T_{n-1}^q +  \frac{1}{q'}\, T_n^q.
    \]
    Hence, it follows that
    \begin{align*}
        \Delta_n &\leq  T_n^q \left( b_n - \rho\gamma_n M_n \right)
        + \rho\gamma_n M_{n-1} T_{n-1} T_{n}^{q-1} \\
        &\leq T_n^q \left( b_n - \rho\gamma_n M_n
        + \frac{q-1}{q}\rho\gamma_n M_{n-1} \right)
        + \frac{\rho\gamma_n M_{n-1}}{q} T_{n-1}^q.
    \end{align*}
    Now, observe that
    \begin{gather*}
        b_n - \rho\gamma_n M_n
        + \frac{q-1}{q}\rho\gamma_n M_{n-1}
        \leq -\frac{\rho\gamma_{n+1} M_n}{q},
        \qquad \forall n\geq 1 \\
        \Updownarrow \\
        \frac{1}{q}\frac{M_n(\gamma_n-\gamma_{n+1})}{b_n}
    +\frac{1}{q'}\frac{\gamma_n(M_n-M_{n-1})}{b_n}
        \geq
        \frac{1}{\rho},
        \qquad \forall n\geq 1 ,
    \end{gather*}
    which is satisfied by hypothesis. Hence, we have
    \begin{align*}
        \Delta_n &\leq T_n^q \left( b_n - \rho\gamma_n M_n
        + \frac{q-1}{q}\rho\gamma_n M_{n-1} \right)
        + \frac{\rho\gamma_n M_{n-1}}{q} T_{n-1}^q \\
        &\leq \frac{\rho\gamma_n M_{n-1}}{q} T_{n-1}^q
        - \frac{\rho\gamma_{n+1} M_n}{q} T_n^q.
    \end{align*}
    Now, this forms a telescopic series. Hence, since $M_0=0$, it follows that
    \begin{align*}
        \sum_{n=1}^s \Delta_n
        &\leq \sum_{n=1}^s
        \left( \frac{\rho\gamma_n M_{n-1}}{q} T_{n-1}^q
        - \frac{\rho\gamma_{n+1} M_n}{q} T_n^q \right) 
        %
        =
        - \frac{\rho\gamma_{s+1}M_s}{q}T_s^q \leq 0.
    \end{align*}
    Hence, since
    \[
        \Delta_n
        =
        b_nT_n^q
        -
        \rho\gamma_nw_n
        \left(\int_{X_n} |f(x)|^p \,\mathrm{d}\mu(x)\right)^{\!\!1/p}
        T_n^{q-1},
    \]
    we have by letting $s\to\infty$ that
    \[
    \sum_{n=1}^\infty b_n T_n^q
    \leq
    \rho \sum_{n=1}^\infty \gamma_n w_n
    \left( \int_{X_n} |f(x)|^p \,\mathrm{d}\mu(x) \right)^{\!\!1/p}
    T_n^{q-1}.
    \]
    Therefore, another application of Hölder's inequality finally implies that
    \begin{align*}
        \sum_{n=1}^\infty b_n T_n^q
        &\leq \rho \sum_{n=1}^\infty \gamma_nw_n
        \!\left(\int_{X_n} |f(x)|^p \,\mathrm{d}\mu(x) \right)^{\!\!1/p}
        T_n^{q-1} \\
        &=
        \rho \sum_{n=1}^\infty
        \left[
        b_n^{1/p}
        \left(\int_{X_n} |f(x)|^p \,\mathrm{d}\mu(x) \right)^{\!\!1/p}
        \right]
        \left[
        b_n^{1/q'}T_n^{q-1}
        \right] \\
        &\leq \rho
        \left(\sum_{n=1}^\infty b_n \int_{X_n} |f(x)|^p \,\mathrm{d}\mu(x) \right)^{\!\!1/p}\!
        \left(
        \sum_{n=1}^\infty
        b_n^{p'/q'}T_n^{(q-1)p'}
        \right)^{\!\!1/p'}.
    \end{align*}
    Since $p\leq q$, we have $\frac{q-1}{q}\geq \frac{1}{p'}.$ Therefore,
    \begin{align*}
        \left(
        \sum_{n=1}^\infty
        b_n^{p'/q'}T_n^{(q-1)p'}
        \right)^{\!\!1/p'}\!
        =
        \left(
        \sum_{n=1}^\infty
        (b_nT_n^q)^{\frac{(q-1)p'}{q}}
        \right)^{\!\!1/p'}\! \leq
        \left(
        \sum_{n=1}^\infty b_nT_n^q
        \right)^{\!\!1-1/q}.
    \end{align*}
    Hence,
    \begin{align*}
        \sum_{n=1}^\infty b_n T_n^q
        &\leq
        \rho
        \left(\sum_{n=1}^\infty b_n \int_{X_n} |f(x)|^p \,\mathrm{d}\mu(x) \right)^{\!\!1/p}\!
        \left(
        \sum_{n=1}^\infty b_nT_n^q
        \right)^{\!\!1-1/q},
    \end{align*}
    from which it finally follows that
    \begin{align*}
         \left( \sum_{n=1}^\infty b_n T_n^q \right)^{\!\!1/q}\!
         \leq
         \rho
         \left(\sum_{n=1}^\infty b_n \int_{X_n} |f(x)|^p \,\mathrm{d}\mu(x) \right)^{\!\!1/p},
    \end{align*}
    as desired.
\end{proof}

\begin{rem}\label{rem - counter-example}
    The condition of \Cref{thm - main} is not necessary. Indeed, consider the case $X=\mathbb{N}$, $X_n=\{n\}$, $\mu$ the discrete measure, $p=q$, $b_n\equiv 1$, $M_n=\alpha^{n}$ and $m_n=\beta^{-n}$ with $\alpha,\beta>1$ and $\frac{\beta}{q}+\frac{1}{\alpha q'}=1$. Then
    \begin{align*}
        \frac{1}{q}\frac{M_n(\gamma_n-\gamma_{n+1})}{b_n}
    +\frac{1}{q'}\frac{\gamma_n(M_n-M_{n-1})}{b_n} = 0,
    \end{align*}
    which implies that $\rho=\infty$. However, 
    \begin{align*}
        \beta &\leq \sup_{N\geq 1} \left( \sum_{n=1}^{\infty}\frac{1}{\alpha^{np}} \right)^{\!\!1/p}\!\! \left(\sum_{k=1}^{\infty} \frac{1}{\beta^{kp'}}
    \right)^{\!\!1/p'}\! = \frac{1}{(\alpha^p-1)^{1/p}(\beta^{p'}-1)^{1/p'}} < \infty.
    \end{align*}
    Hence, \Cref{cor - charac} allows us to conclude that 
    \begin{align*}
        \left(\sum_{n=1}^{\infty} \Bigg|\frac{1}{M_n}\sum_{k \in \mathbf{N}_n} m_k a_k \Bigg|^p\right)^{\!\!1/p}
        \!\!\leq  \frac{p^{1/p}(p')^{1/p'}}{(\alpha^p-1)^{1/p}(\beta^{p'}-1)^{1/p'}}\!\left(\sum_{n=1}^\infty|a_n|^p\right)^{\!\!1/p}\!,
    \end{align*}
    while \Cref{thm - main} does not.
\end{rem}

Now, as we did previously, let us consider some important particular cases for which we recover some results in the literature.

\begin{enumerate}

    \item \emph{The general discrete Hardy inequality.} Choose $X=\mathbb{N}$, $X_n=\{n\}$, $\mu$ is the counting measure, $p=q$, $m_n=b_n$ for all $n\geq 1$ and $M_n=m_1+\cdots+m_n$. Then $w_n = b_n$ for all $n\geq 1$ and the sufficient condition once again becomes
    \[
    \rho \geq \frac{p}{p-1} \qquad \text{for }~ n=1,2,3,\dots
    \]
    Hence, choosing $\rho=\frac{p}{p-1}$ recovers the sharp constant in the weighted discrete Hardy inequality, due to Copson \cite{Copson1928}. The discrete Hardy inequality corresponds to the case $m_k\equiv 1$. Note that the results from both \cite{Bouthat2023} and \cite{VincentSohani2025} did not have the classical Hardy inequality as a special case. 

    \item \emph{The geometric Hardy inequality.} Let $p=q=2$, $X=\mathbb{N}$, $\mu$ the discrete measure, $X_1=\{1,\dots,b\}$ and $X_n=\{b^{n-1}+1,\dots,b^{n}\}$ for $n\geq 2$, where $b\geq 2$ is an integer. Moreover, let $b_n\equiv 1$, $m_n=1$ and $M_n = b^{n/2}$. Then, if $n\geq 2$, we have 
    \begin{align*}
        \frac{1}{q}\frac{M_n(\gamma_n-\gamma_{n+1})}{b_n}
    +\frac{1}{q'}\frac{\gamma_n(M_n-M_{n-1})}{b_n} = \frac{\sqrt{b}-1}{\sqrt{b-1}}
    \end{align*}
    and
    \[
    \frac{1}{q}\frac{M_n(\gamma_n-\gamma_{n+1})}{b_n}
    +\frac{1}{q'}\frac{\gamma_n(M_n-M_{n-1})}{b_n} = 1-\frac{1}{2\sqrt{b-1}} \geq \frac{\sqrt{b}-1}{\sqrt{b-1}}
    \]
    if $n=1$. Hence, the optimal value for $\rho^2$ is $\bigl(\frac{\sqrt{b}-1}{\sqrt{b-1}}\bigr)^{-2} = \frac{\sqrt{b}+1}{\sqrt{b}-1}$, which yields
    \begin{equation*}\label{eq - expo}
         \sum_{k=1}^{\infty} \frac{1}{b^k}  \bigg| \sum_{j=1}^{b^k} a_j \bigg|^2
        \leq~
        \frac{\sqrt{b}+1}{\sqrt{b}-1} \sum_{n=1}^{\infty} |a_n|^2.
    \end{equation*}
    This is precisely \cite[Theorem 2]{Bouthat2023}, where it is shown that this inequality is sharp.

    \item \emph{Example 1 in \cite{Bouthat2023}}. Consider the case $p=q=2$, $X=\mathbb{N}$, $\mu$ the discrete measure, $b_n\equiv1$, $X_n=\{2^{n-1},\dots,2^n-1\}$, $M_n = \sum_{k=1}^n k\cdot 2^{(k-1)/2}$ and $m_n$ is the sequence which takes the value $n$ for all indices in $X_n$. Then $w_k=k\cdot 2^{(k-1)/2}$, and it is shown in \cite{Bouthat2023} that the bound holds with the constant $\rho \approx 1.2114$. However, 
    \begin{align*}
        \frac{1}{q}\frac{M_n(\gamma_n-\gamma_{n+1})}{b_n}
    +\frac{1}{q'}\frac{\gamma_n(M_n-M_{n-1})}{b_n}
        &= \frac{1}{2} \!\bigg(1+2^{\frac{n}{2}}\bigg(\frac{1}{n}-\frac{1}{\sqrt{2}(n+1)}\bigg)\sum_{k=1}^{n}k\,2^{\frac{k}{2}} \bigg) \\
        &\geq \frac{1}{2} \!\bigg(1+2^{\frac{1}{2}}\bigg(\frac{1}{1}-\frac{1}{\sqrt{2}(1+1)}\bigg)\sum_{k=1}^{1}k\,2^{\frac{k}{2}} \bigg) \\
        &= \frac{6-\sqrt{2}}{4}.
    \end{align*}
    Hence, \Cref{thm - main} ensures that we can take the constant $\rho = \frac{4}{6-\sqrt{2}} \approx 0.87226$, which is a significant improvement on \cite[Example 1]{Bouthat2023}.
\end{enumerate}

The above examples might make one think that the inequality of \Cref{thm - main} is always better than the one of \cite[Theorem 1]{Bouthat2023}. However, this is not the case, as the following example shows.

\begin{example}
    Consider the case $X=\mathbb{N}$, $\mu$ the discrete measure, $X_n=\{n\}$, $p=q$, $b_n\equiv 1$, $m_n=2^n$ and $M_n=\sum_{k=1}^n w_k = 2(2^n-1)$ with $1<p\leq \frac{5}{4}$. Then $w_n=m_n=2^n$ and \Cref{thm - main} yields 
    \begin{align*}
        \frac{1}{q}\frac{M_n(\gamma_n-\gamma_{n+1})}{b_n}
    +\frac{1}{q'}\frac{\gamma_n(M_n-M_{n-1})}{b_n}  = 1-\frac{1}{p2^{n}},
    \end{align*}
    and we can take $\rho = \frac{1}{1-1/2p}$. On the other hand, the constant given by \cite[Theorem 1]{Bouthat2023} (see \eqref{E:3}) is
    \begin{align*}
        \sup_{n\geq 1} \left(w_n\sum_{k=n}^\infty \frac{1}{M_n} \right)^{\!\!1/p}\!\! = \sup_{n\geq 1} \left(\sum_{k=0}^\infty \frac{1}{2(2^{k}-2^{-n})} \right)^{\!\!1/p}\!\! = \left(\sum_{k=1}^\infty \frac{1}{2^{k}-1}\right)^{\!\!1/p}\! \lessapprox 1.606695.
    \end{align*}
    Now, since $p\leq \frac{5}{4}$, we have $\rho = \frac{1}{1-1/2p} \geq \frac{1}{1-1/(2\cdot \frac{5}{4})} = \frac{5}{3} > 1.606695$. Hence, in this case \cite[Theorem 1]{Bouthat2023} gives a sharper constant than \Cref{thm - main}.
\end{example}

Let us make one last remark to conclude this section. Recall that \Cref{rem - counter-example} showed that there are some cases for which \cite[Theorem 1]{Bouthat2023} is able to provide an inequality (albeit not always with a sharp constant) while \Cref{thm - main} is not. However, it turns out that such a phenomenon only happens in some specific cases. Indeed, it turns out that in one of the more natural case, namely the case $X=\mathbb{N}$, $\mu$ the discrete measure, $p=q$, $b_n\equiv 1$ and $M_n=w_1+\cdots + w_n$, which is the subject of \cite[Theorem 1]{Bouthat2023}, our result subsumes the latter whenever it applies. 

More precisely, in this case, the sufficient condition of \Cref{thm - main} simplifies to
\begin{equation}\label{eq - reduced}
    \frac{1}{p}\left(\frac{1}{w_n}-\frac{1}{w_{n+1}}\right)\!M_n+\frac{1}{p'} \geq \frac{1}{\rho} \qquad \text{for }~ n=1,2,3,\dots
\end{equation}
Now, we claim that in this case, if
\[
\frac{1}{p}\left(\frac{1}{w_n}-\frac{1}{w_{n+1}}\right)\!M_n+\frac{1}{p'} > 0 \qquad \text{for }~ n=1,2,3,\dots,
\]
and if the estimate of the constant in \Cref{thm - main} diverges, then the estimate of \cite[Theorem 1]{Bouthat2023} is infinite. In other words, in this situation, the global condition of \cite[Theorem 1]{Bouthat2023} induces the local restriction of \Cref{thm - main}.



\begin{prop}\label{prop - infty}
    Consider the hypothesis of \Cref{thm - main} with $X=\mathbb{N}$, $\mu$ the discrete measure, $p=q$, $b_n\equiv 1$ and $M_n=w_1+\dots+w_n$. Moreover, suppose that 
    \begin{equation}
        \frac{1}{p}\left(\frac{1}{w_n}-\frac{1}{w_{n+1}}\right)\!M_n+\frac{1}{p'} > 0 \qquad \text{for }~ n=1,2,3,\dots 
    \end{equation}
    If $\rho=\infty$, then 
    \[
    \rho' = \sup_{n\geq 1} \left( w_n \sum_{j=n}^{\infty} \frac{1}{M_j} \right) = \infty.
    \]
\end{prop}
\begin{proof}
    Since we suppose that 
    \[
    \frac{1}{p}\left(\frac{1}{w_n}-\frac{1}{w_{n+1}}\right)\!M_n+\frac{1}{p'} > 0 \qquad \text{for }~ n=1,2,3,\dots,
    \]
    and since $\rho$ is such that 
    \[
    \frac{1}{p}\left(\frac{1}{w_n}-\frac{1}{w_{n+1}}\right)\!M_n+\frac{1}{p'} \geq \frac{1}{\rho} \qquad \text{for }~ n=1,2,3,\dots,
    \]
    then 
    \begin{align}\label{eq - liminf}
        \rho=\infty \quad\iff\quad  \,\,&\frac{1}{p'} + \frac{1}{p} \liminf_{n\to\infty} \left(\frac{1}{w_n}-\frac{1}{w_{n+1}}\right)\!M_n =0 \nonumber\\
        \iff\quad  \,\,&\liminf_{n\to\infty} \left(\frac{1}{w_{n+1}}-\frac{1}{w_n}\right)\!M_n = p-1.
    \end{align} 
    Denote by $\ell\in [0,\infty]$ the quantity $\ell:= \limsup_{n\to\infty}w_{n}$. If $\ell=\infty$, then $w_{n+1} > w_n$ infinitely often, which would imply that $\frac{1}{w_{n+1}}-\frac{1}{w_n} < 0$ infinitely often and contradict \eqref{eq - liminf}. Therefore, we must have $\ell<\infty$. 
    Since $\limsup_{n\to\infty} w_n = \ell$, there exists an $N_\varepsilon\in \mathbb{N}$ such that $w_n < \ell+\frac{\varepsilon}{2}$ for all $n\geq N_\varepsilon$. Hence, we have for all $n$ sufficiently large (without loss of generality, suppose that $n\geq N_\varepsilon$ suffice) that 
    \begin{align*}
        M_n &= w_1+\cdots +w_{N_\varepsilon-1} + w_{N_\varepsilon}  \cdots + w_n <  K_\varepsilon + (n-N_\varepsilon) \Bigl(\ell+\frac{\varepsilon}{2}\Bigr) \leq (\ell+\varepsilon) n
    \end{align*}
    Therefore, it finally follows that
    \begin{align*}
        \sup_{n\geq 1} \left( w_n \sum_{j=n}^{\infty} \frac{1}{M_j} \right) &\geq w_{N_\varepsilon} \sum_{j=N_\varepsilon}^{\infty} \frac{1}{M_j} > \frac{w_{N_\varepsilon}}{\ell+\varepsilon}  \sum_{j=N_\varepsilon}^{\infty} \frac{1}{j} = \infty,
    \end{align*} 
    as desired.
\end{proof}

\section{Concluding remarks}\label{sec - conclusion}

In this paper, we generalized and characterized a class of Hardy-type inequalities in which the usual arithmetic means are replaced by weighted averages over nested subsets of a measure space. We did so using two different approaches, both distinct from the one previously used in the literature. 
Moreover, using some examples, we showed that the local criterion does not dominate the earlier global conditions in every possible situation. Thus the two approaches are complementary: the characterization gives a complete boundedness criterion, while the local condition often gives sharper constants in structured cases. 

Several natural questions remain open. First, it would be interesting to determine whether the present framework can be extended beyond $L^p$ spaces to more general function spaces, such as weighted $L^p$ spaces, Lorentz spaces, Orlicz spaces, or other Banach function spaces. Since the proof of the reduction relies heavily on Hölder's inequality and duality between $L^p$ and $L^{p'}$, such extensions would likely require identifying the correct analogue of the quantities $w_n$. 

A second direction concerns kernel operators. The inequalities considered here can be viewed as boundedness results for triangular averaging kernels. This suggests studying more general operators of the form 
\[ 
Tf(x)=\int_X k(x,y)f(y)\,\mathrm{d}\mu(y) 
\] 
and asking for conditions under which $T$ is bounded on $L^p$. Some cases may be accessible through Schur-type tests, but the classical Hardy operator already shows that Schur's test does not capture the full picture. It would therefore be useful to develop criteria adapted to Hardy-type kernels, especially those arising from nested sets or monotone integration regions. 

Finally, one could also consider versions in which the averages are replaced by sums of integrals over several pieces of a partition, or even over a sequence of different measure spaces. Such formulations would include inequalities of the form 
\[ 
\left(\sum_{n=1}^{\infty} b_n \bigg| \frac{1}{M_n} \sum_{k=1}^{n} \int_{X_k} m_k(x)f_k(x)\,\mathrm{d}\mu_k(x) \bigg|^q\right)^{\!\!1/q}\! \leq \rho \left( \sum_{n=1}^{\infty} b_n \int_{X_n}|f_n(x)|^p\,\mathrm{d}\mu_n(x) \right)^{\!\!1/p}. 
\] 
The methods of this paper suggest that such inequalities should again be closely related to discrete weighted Hardy inequalities, with the quantities $w_k$ replaced by suitable $L^{p'}$-norms of the functions $m_k$. Making this precise would provide a flexible framework that contains both the discrete and measure-theoretic inequalities treated here. 

These questions suggest that the inequalities studied in this paper can be viewed as a part of an even broader theory of weighted Hardy operators, nested averaging operators, and triangular kernel operators. The results obtained here thus give a foundation for that theory by separating the boundedness problem, which admits a complete characterization, from the sharper constant problem, for which local and structural methods appear to still be essential.

    \bibliographystyle{plain}
    \bibliography{ref}

\end{document}